\documentclass[amsmath,amssymb,twoside,epstopdf,12pt]{article}
\usepackage{amsfonts}
\usepackage{graphicx}
\usepackage{subfigure}
\usepackage{setspace}
\usepackage{subfig}
\usepackage{array}
\usepackage[margin=1in]{geometry}
\usepackage{amsmath, amsfonts, amssymb, amsthm}
\newtheorem{thm}{Theorem}[section]

\theoremstyle{definition}

\theoremstyle{remark}

\numberwithin{equation}{section} \hfuzz3pt
\begin{document}

\setcounter{page}{1} \vspace{0.1cm}
\begin{center}
{\large{\bf An Extension of Clark-Haussman Formula and Applications   }}\\
 \vspace{0.4cm}

{\bf Traian A. Pirvu$^1$ and Ulrich G. Haussmann$^2$ }\\
$^1$ Department of Mathematics and Statistics, McMaster University\\
1280 Main Street West, Hamilton, ON, L8S 4K1\\
tpirvu@math.mcmaster.ca\\
$^2$  Department of Mathematics, The University of British Columbia\\ Vancouver, BC, V6T1Z2  \\uhaus@math.ubc.ca
 \vspace{0.4cm}

\end{center}

\begin{abstract}
This work considers a stochastic model in which the uncertainty is driven by a multidimensional Brownian motion.
The market price of risk process makes the transition between real world probability measure and risk neutral probability measure. 
Traditionally, the martingale representation formulas under the risk neutral probability measure requires the market price of risk process
to be bounded. However, in several financial models the boundedness assumption of the market price of risk fails. One example is
a stock price model with  the market price of risk following an Ornstein-Uhlenbeck process. This work extends Clark-Haussmann formula to underlying stochastic processes which fail to satisfy the standard requirements. Our result can be applied to hedging and optimal investment in stock markets with unbounded market price of risk.
\end{abstract}

\section{Introduction}
Clark-Haussmann-Ocone formula expresses the integrand from the Martingale Representation Theorem 
in terms of Fr\'echet (Clark-Haussmann) or Malliavin's (Ocone) derivative. 
In financial applications the Martingale Representation Theorem is performed
under the martingale measure. Within this setting Ocone and Karatzas \cite{KAR2} specializes
Clark-Haussmann-Ocone formula to obtain
\begin{equation}\label{=}
V(T)={\tilde{E}}V(T)+\int_{0}^{T}{\tilde{E}}\left[\left(D_{t}V(T)-V(T)\int_{t}^{T}D_{t}\tilde{\theta}(u)d\tilde{W}(u)\right)\bigg\vert\mathcal{F}(t)\right]d\tilde{W}(t). 
\end{equation}
Here $V(T)$ is a random variable on a filtered probability space
$(\Omega,\{\mathcal{F}_t\}_{0\leq t\leq T},\mathcal{F},P),$ $D_{t}V(T)$ is Malliavin derivative, and $\tilde{W}$ is a Brownian motion under the martingale probability
measure $Q.$ This probability measure is defined by $Q(A):=E[\tilde{Z}(T){\bf{1}}_{A}],$ and $ \tilde{Z}(t):=\exp\left\{-\int_{0}^{t}\tilde{\theta}(u)\,dW(u)-\frac{1}{2}\int_{0}^{t} \tilde{\theta}^2(u)\,du\right\};$ $\tilde{\theta}$ is the market price of risk, and ${\tilde{E}}$ is the expectation operator under $Q.$

The representation formula (\ref{=}) is important in pricing and hedging of derivative securities. For instance if $V(T)$ is the payoff of a derivative security
 formula (\ref{=}) yields the hedging portfolio. The representation formula (\ref{=}) can be applied in portfolio management. Indeed the problem of finding optimal portfolios can be reduced by the martingale method to the problem of hedging a derivative security.

The representation formula (\ref{=}) works under the boundeness assumption of the market price of risk process $\tilde{\theta}$
 (see Theorem $2.5$ in \cite{KAR2}). In some models (see Kim and Omberg \cite{KIM}), $\tilde{\theta}$ is an Ornstein-Uhlenbeck process which fails to remain bounded. 
 
 {\bf The contribution of this paper is the extension of (\ref{=}) to models with unbounded $\tilde{\theta}.$} In particular our extension cover the case in which  $\tilde{\theta}$  is an Ornstein-Uhlenbeck process.

The remainder of this paper is structured as follows. Section $2$ presents the model. Section $3$ contains the main result.  We conclude with an appendix containing the proof of the main result.

\section{Model Description}

Let $T>0$ be a finite deterministic horizon. Here $W=(W_{1},\cdots,W_{n})^{T}$ is a $n-$dimensional Brownian motion on a filtered probability space
$(\Omega,\{\mathcal{F}_t\}_{0\leq t\leq T},\mathcal{F},P),$ where $\{\mathcal{F}_t\}_{0\leq t\leq T}$ is the completed filtration generated by $W.$ Let $\tilde{\theta}=(\tilde{\theta}_{1},\cdots, \tilde{\theta}_{n})^{T}$ be a $n-$dimensional adapted stochastic process. We assume that
\begin{equation}\label{narde} E\left[\exp\left(\frac{1}{2}\int_{0}^{T}\parallel\tilde{\theta}(u)\parallel^{2}\,du\right)\right]<\infty, \end{equation}
 where as usual $||\cdot||$ denotes the Euclidean norm in ${R}^{n}.$ One can recognize this as the Novikov condition and it is sufficient to ensure that the
 stochastic exponential process 
 
  \begin{equation}\label{Nov}
  \tilde{Z}(t)=Z_{\tilde{\theta}}(t):=\exp\left\{-\int_{0}^{t}\tilde{\theta}^{T}(u)\,dW(u)-\frac{1}{2}\int_{0}^{t}\parallel\tilde{\theta}(u)\parallel^{2}\,du\right\}\end{equation}
is a (true) martingale. Moreover by the Girsanov theorem ( see Section $3.5$ in \cite{KAR}) \begin{equation}\label{Br}\tilde{W}(t)=W(t)+\int_{0}^{t}\tilde{\theta}(u)\,du \end{equation} is a Brownian motion under the equivalent martingale measure ${\tilde{Q}}$ defined by 

  \begin{equation}\label{Q}
  {\tilde{Q}}(A):=E[\tilde{Z}(T){\bf{1}}_{A}],\qquad A\in \mathcal{F}_T.\end{equation}
  Below we shall have occasion to write the process $\tilde{\theta}(t),$ as a function of the process $W(t),$ i.e., $\tilde{\theta}(t,\omega)=\tilde{\Theta}(t,W(\cdot,\omega))$ a.s. 
Now define a mapping $\bar{\Theta}$ of $(C[0,T])^{n}$ into $(L^{\infty}[0,T])^{n}$ by $\bar{\Theta}(y)(t):=\tilde{\Theta}(t,y(\cdot)).$ Then $\bar{\Theta}$ is \textit{nonanticipative} in the sense that $\bar{\Theta}(y)(t)=\bar{\Theta}(z)(t)$
 for $y, z$ such that $y(s)=z(s)$ on $0\leq s\leq t$; this is equivalent to demanding $\tilde{\Theta}(t,W(\cdot))$ be $\{\mathcal{F}_t\}$ -adapted. We assume that
 \begin{equation}\label{as1}
 \sup_{t}\parallel\tilde{\Theta}(t,0)\parallel<\infty,
 \end{equation} 
and $\bar{\Theta}$ is Frech\'et differentiable with derivative $\bar{\Theta}'(y),$ i.e. for $y\in(C[0,T])^{n},$
 $\bar{\Theta}'(y)$ is a bounded linear operator mapping  $(C[0,T])^{n}$ into $(L^{\infty}[0,T])^{n}$ such that $$\parallel \bar{\Theta}(y+h)-\bar{\Theta}(y)-\bar{\Theta}'(y)h \parallel_{\infty}=o(\parallel h \parallel_{T}),$$  where $\parallel\cdot\parallel_{T}$ is the norm in $(C[0,T])^{n},$ i.e. $\parallel y\parallel_{T}=\sup_{0\leq t\leq T}\parallel y(t)\parallel,$ and $\parallel \cdot\parallel_{\infty}$ is the norm in $(L^{\infty}[0,T])^{n}.$ 
 
 The Riesz Representation Theorem gives, for fixed $t,$ the existence of a unique finite signed measure $\tilde{\mu}$  such that
  \begin{equation}\label{meas}[\bar{\Theta}'(y)h](t)=\int_{0}^{t}\tilde{\mu}(ds,y,t)h(s).\end{equation}

 We require that
 for some $\delta>0$ and constant $K_{\delta}$
 \begin{equation}\label{as2}
 |\bar{\Theta}'(y_1)-\bar{\Theta}'(y_2)|_{t}\leq
 K_{\delta}\parallel y_1-y_2 \parallel_{T}^{\delta}
 ,\end{equation}
 and
 
 \begin{equation}\label{as3}
  \sup_{y,t}|\bar{\Theta}'(y)|_{t}=\sup_{y,t}\mathrm{var}_{[0,t]}(\tilde{\mu}(\cdot,y,t))<\infty,  
 \end{equation}
 where
 $$|\bar{\Theta}'(y)|_{t}=\mathrm{var}_{[0,t]}(\tilde{\mu}(\cdot,y,t)). $$
 
 Recall that \begin{equation}\label{10o}dW(t)=d\tilde{W}(t)-\tilde{\Theta}(t,W(\cdot))dt=d\tilde{W}(t)-\bar{\Theta}(W)(t)dt,\end{equation} and by the above assumptions this SDE
 (where the unknown process is $W$) has a unique solution $W$ (see Theorem $6,$ page $249$ in \cite{Protter}). Hence $W$ and $\tilde{W}$ generate the same filtration
 $\{\mathcal{F}_t\}_{0\leq t\leq T}.$
 
 In what follows we denote by ${\tilde{E}}$ the expectation operator with respect to the probability measure ${\tilde{Q}}.$

\section{Main Result}
Let $V(T)\in L^{2}[0,T]$ be an $\mathcal{F}_T$ adapted random variable. The martingale ${\bar{V}}(t):=\tilde{E}[V(T)|\mathcal{F}_{t}],$\,\,admits the stochastic integral representation
 \begin{equation}\label{8i}
\bar{V}(t)=\tilde{E}V(T)+\int_{0}^{t}\beta^{T}(u)d\tilde{W}(u),\qquad 0\leq t\leq T,
\end{equation}
 for some $\mathcal{F}_t-$adapted process $\beta$ which satisfies $\int_{0}^{T} ||\beta(u)||^{2}\,du<\infty$ a.s. (e.g., \cite{KAR}, Lemma $1.6.7$).
The process $\beta(t)$ of (\ref{8i}) can be computed explicitly by Haussmann's formula.  Let ${V}(T)=L(\tilde{Z},W),$ where $L:C[0,T]\times(C[0,T])^{n}\longrightarrow{R}.$ Assume further that
   \begin{eqnarray}\label{pis}
   L'(z,w)(v_1,v_2)&=&\int_{0}^{T}\mu(dt,z,w)(v_1(t),v_2(t)^{T})^{T},\,\,\,\mbox{a.s.}
   \end{eqnarray}
  for some finite signed measure ${\mu},$ and  
    \begin{equation}\label{sid} |L'(y_1)-L'(y_2)|\leq K(1+\parallel y_1\parallel_{T}^{\beta})(1+\parallel y_2 \parallel_{T}^{\beta})\parallel y_1-y_2\parallel_{T}^{\rho},\end{equation} for some positive $K, \beta, \rho.$ Let $Y=(\tilde{Z},W),$ so
 $$Y(t)=Y(0)+\int_0^t f(u,Y(\cdot))\,du+\int_0^t g(u,Y(\cdot))d\tilde{W}(u),\quad 0\leq t\leq T. $$
  Here with $y=(y_1,y_2),$ $y_1$ a scalar process and $y_2$ a $n$-dimensional process, 
 $$f(u,y)=\left( y_1\parallel\tilde{\Theta}(u,y_2)\parallel^2, \,\,\, -\tilde{\Theta}(u,y_2) \right)^T $$
 $$g(u,y)=\left( -y_1\tilde{\Theta}^{T}(u,y_2),\,\,\, I_n \right)^T, $$
 where $I_n$ is the $n\times n$ identity matrix. Let $\Phi(t,s)$ be the unique solution of the linearized equation 
$$d\Phi(t,s)=\left[\frac{\partial f}{\partial y}(t,Y(\cdot))\Phi(t,s)\right](t)dt+\left[\frac{\partial g}{\partial y}(t,Y(\cdot))\Phi(t,s)\right](t)d\tilde{W}(t),\qquad t>s,$$
   $$\Phi(s,s)=I_{n+1},\,\,\mbox{and}\,\,\,\,\Phi(t,s)=O_{n+1}\,\,\,\,\mbox{for}\,\,\,\,0\leq t<s.$$
     At this point one may wonder about the existence and uniqueness of $\Phi.$ The matrix  process $\Phi$ has $j-$th column $\Phi^{j}=(\Phi^{1,j} ,(\Phi^{2,j})^{T})^{T}.$
The scalar process $\Phi^{1,j}$ satisfies
\begin{eqnarray*}
d\Phi^{1,j}(t,s)&=&\left(\parallel\bar{\Theta}(W)(t)\parallel^2\Phi^{1,j}(t,s)+2\tilde{Z}(t)\bar{\Theta}^{T}(W)(t)[\bar{\Theta}'(W)\Phi^{2,j}(\cdot,s)](t)\right)dt\\ &  & \mbox{} - \left(\bar{\Theta}^{T}(W)(t)\Phi^{1,j}(t,s)+\tilde{Z}(t)[\tilde{\Theta}'(W)\Phi^{2,j}(\cdot,s)](t)\right)d\tilde{W}(t).
\end{eqnarray*}
The $n-$dimensional vector process $\Phi^{2,j}$ satisfies
\begin{equation}\label{55}
\frac{d\Phi^{2,j}(t,s)}{dt}=-\left[\bar{\Theta}'(W)\Phi^{2,j}(\cdot,s)\right](t)=-\int_{s}^{t}\tilde{\mu}(du,W,t)\Phi^{2,j}(u,s),
\end{equation}
where $\tilde{\mu}$ is the measure defined in (\ref{meas}). Let us notice that $\Phi^{2,1}\equiv0.$  For $j>1$ and some constant $K,$
\begin{eqnarray*}
\parallel\Phi^{2,j}(t,s)-e_{j-1} \parallel&=&\parallel\Phi^{2,j}(t,s)-\Phi^{2,j}(s,s) \parallel\\&=&\parallel\int_{s}^{t}-[\bar{\Theta}'(W)\Phi^{2,j}(\cdot,s)] (u)\,du\parallel\\&\leq&\int_{s}^{t}K\sup_{s\leq v\leq u}\parallel\Phi^{2,j}(v,s)\parallel\,du 
,\end{eqnarray*} where  $e_{i}$ is the $i^{th}$ column of $I_{n}$ and the last inequality comes from (\ref{as3}). Therefore
$$ \sup_{s\leq v\leq t}\parallel\Phi^{2,j}(v,s)\parallel\leq 1+K\int_{s}^{t}\sup_{s\leq v\leq u}\parallel\Phi^{2,j}(v,s)\parallel\,du ,$$ hence by Gronwall's inequality
\begin{equation}\label{oop}
 \sup_{s\leq v\leq t}\parallel\Phi^{2,j}(v,s)\parallel\leq e^{K(t-s)}.
\end{equation}
Existence and uniqueness of the process $\Phi$ it is now straightforward. Let us define
 \begin{equation}\label{12}\lambda(t,\omega):=\left[\int_{t}^{T}\mu(du,\tilde{Z}(\omega),W(\omega))\Phi(u,t,\omega)\right]g(t,(\tilde{Z}(\omega),W(\omega))), \end{equation}
 with $\mu$ the measure of (\ref{pis}).

 \begin{thm}\label{main22}
  (Clark-Haussmann formula)
 \begin{equation}\label{23}L(\tilde{Z}(\cdot),W(\cdot))=\int_{0}^{T}{\tilde{E}}(\lambda(t)\vert\mathcal{F}_{t})d\tilde{W}(t)+{\tilde{E}}L[\tilde{Z}(\cdot),W(\cdot)]. \end{equation}
  \end{thm}

\begin{proof}

See the appendix.

\end{proof}

\subsection{ An Example}
 Consider a stochastic volatility model, in which the market price of risk follows an Ornstein-Uhlenbeck process.
 This model is considered by both \cite{KIM} and \cite{Wachter}. 
 In this example, the financial market consists of one bond and one stock whose price $S$ is given by
 $$\frac{dS(t)}{S(t)} =\mu(t)\,dt
+\sigma(t)\,dW(t).$$ Let $U$ be the Ornstein-Uhlenbeck process $$dU(t)=(\alpha-\beta U(t))dt+vdW(t),$$
In this model
 \begin{equation}\label{3z}\tilde{\theta}(t)=U(t)\end{equation} which is an unbounded process. Condition (\ref{as1}) is obviously satisfied and one can prove that (\ref{narde}) hold for $T$ below some thereshold. 
In fact  \begin{equation}\label{e3}U(t)=e^{-\beta t}U(0)+\left(\frac{\alpha}{\beta}+\frac{v^{2}}{2\beta}\right)(1-e^{-\beta t})+ve^{-\beta t}\int_{0}^{t}e^{\beta u}dW(u),\end{equation}
whence $$ [\bar{\Theta}'(W)(\gamma)](t)  =v\left[\gamma(t)-\beta\int_{0}^{t}e^{\beta( u-t)}\gamma(u)\,du\,\right],$$ so (\ref{as2}) and (\ref{as3}) hold.

Next, assume that the mean variance optimization problem is considered within this paradigm. It can be shown by using the martingale approach (see page 165 in \cite{CZ}) that the optimal wealth is
$$ V(T)=\lambda_1 + \lambda_2 \tilde{Z}(T),$$ 
for some Lagrange multipliers $\lambda_1, \lambda_2$ which can be computed. In order to find the mean variance optimal portfolio one has to apply Theorem 1 to 
$$L(\tilde{Z}(\cdot),W(\cdot))=\lambda_1 + \lambda_2 \tilde{Z}(T).$$ Notice that the assumptions (\ref{pis}), (\ref{sid}) are met in this case.

\section*{Appendix}
\addcontentsline{toc}{section}{Appendix}

\noindent {\bf{Proof of Theorem \ref{main22}}}: Let us notice that we cannot apply Clark-Haussmann formula right away because hypothesis $H_3$ in \cite{HA} fails. However we can get around this by an approximation argument. Let $\phi_k$ be a sequence of bounded differentiable
 functions on ${R}$ with H\"older continuous (of order $\delta$  (see (\ref{as2}))) derivatives, such that $\phi_k(x)=x,$ if $|x|\leq k,$\,\,  $|\phi_k(x)|\leq |x|,$\,\, $|\phi'_k(x)|\leq1,$ and define
 $$\phi_k(\tilde{\theta}):=(\phi_k(\tilde{\theta_{1}}),\cdots,\phi_k(\tilde{\theta_{n}})),$$
 
 $$f_{k}(u,y)=\left( \phi_k(y_1)\parallel\phi_{k}(\tilde{\theta}(u,y_2))\parallel^2,\,\,\,-\tilde{\theta}(u,y_2)\right)^T,$$$$ g_{k}(u,y)=\left(-\phi_k(y_1)\phi_k(\tilde{\theta}(u,y_2)), \,\,\,I_n\right). $$

 It is easily seen that for fixed $k,$ the functions $f_k(\cdot,\cdot)$ and $g_k(\cdot,\cdot)$ satisfy
 the conditions of Theorem\,$1$ in \cite{HA}. Let us denote by $Y_{k}$ the unique strong solution of the SDE 
 $$dY(t)=f_{k}(t,Y(\cdot))dt+g_{k}(t,Y(\cdot))d\tilde{W}(t),$$ (the existence of a unique strong solution follows since
  $f_k(\cdot,\cdot)$ and $\sigma_k(\cdot,\cdot)$ are Lipschitz). In fact $Y_{k}(t)=(\tilde{Z}_{k}(t),W(t))$ for some process $\tilde{Z}_{k}.$ It turns out that the process  $\tilde{Z}_{k}$ is strictly positive.
  Indeed when $\tilde{Z}_{k}$ gets close to zero it satisfies 
   $$d\tilde{Z}_{k}(t)=\tilde{Z}_{k}(t)(\parallel\phi_{k}(\tilde{\theta})\parallel^2\,dt-\phi_{k}(\tilde{\theta})d\tilde{W}(t)) ,$$ hence the positivity.

   Following \cite{HA}, let $\Phi_k(t,s)$ be the $(n+1)\times (n+1)$ matrix which solves the linearized equation $$dZ(t)=\left[\frac{\partial f_{k}}{\partial x}(t,Y_{k}(\cdot))Z(\cdot)\right](t)dt+\left[\frac{\partial g_k}{\partial x}(t,Y_{k}(\cdot))Z(\cdot)\right](t)d\tilde{W}(t),\qquad t>s,$$
 $$\Phi_{k}(s,s,\omega)=I_{n+1},\,\,\mbox{and}\,\,\,\,\Phi_{k}(t,s,\omega)=O_{n+1}\,\,\,\,\mbox{for}\,\,\,\,0\leq t<s,$$ with $O_{n+1}$ the $(n+1)\times(n+1)$ matrix with zero entries.
 Next with $\mu$ of (\ref{pis}) we define $$\lambda_{k}(t,\omega):=\left[\int_{t}^{T}\mu(du,(\tilde{Z}_{k}(\omega),W(\omega)))\Phi_k(u,t,\omega)\right]g_{k}(t,(\tilde{Z}_{k}(\omega),W(\omega))). $$ Theorem $1$ in \cite{HA} gives the following representation (Haussmann's formula)
\begin{equation}\label{09}L(\tilde{Z}_{k}(\cdot),W(\cdot))=\int_{0}^{T}{\tilde{E}}(\lambda_{k}(t)\vert\mathcal{F}_{t})d\tilde{W}(t)+{\tilde{E}}L[\tilde{Z}_{k}(\cdot),W(\cdot)]\end{equation}
 By using Burkholder-Davis-Gundy and Gronwal inequalities one can get that  for every real number $r$
\begin{equation}\label{ris}
{\tilde{E}}\sup_{0\leq t\leq T}\tilde{Z}_{k}^{r}(t)<\infty,
\end{equation}
uniformly in $k$ and
\begin{equation}\label{2}
{\tilde{E}}\sup_{0\leq t\leq T}\tilde{Z}^{r}(t)<\infty.
\end{equation}
With the notations $a\vee b:=\max(a,b)$
 and $|\tilde{\theta}|:=\max\{|\tilde{\theta}_{i}|:\,i=1,\cdots,n\},$ let us define the following sequence of stopping times
  $$\tau_{k}:=\inf\{s\leq T,\,\,\mbox{such\,that},\,\,\tilde{Z}_{k}(s)\vee|\bar{\Theta}(W)(s)|\geq k\} $$
We claim that
 $\tilde{Z}_{k}(t)=\tilde{Z}(t)$ for $t\leq\tau_{k}.$ Moreover $\tau_{k}\uparrow T$\,\,$P$\,a.s. 
Indeed, let us notice that on $[0,\tau_{k}],$\, $f_k=f$ and $g_k=g,$ hence $\tilde{Z}_{k}(t)=\tilde{Z}(t)$ (since it satisfies the same SDE ). Therefore \begin{eqnarray}\tau_{k}&:=&\inf\{s\leq T,\,\,\mbox{such\,that},\,\,\tilde{Z}_{k}(s)\vee|\bar{\Theta}(W)(s)|\geq k\}\\&=&\inf\{s\leq T,\,\,\mbox{such\,that},\,\,\tilde{Z}(s)\vee|\bar{\Theta}(W)(s)|\geq k\}.\end{eqnarray} Thus for $t\leq\tau_k,$ one has $\tilde{Z}(t)\vee|\bar{\Theta}(W)(t)|\leq k,$ so $t\leq\tau_{k+1}.$
   This proves $\tau_{k+1}\geq\tau_{k}.$ For a fixed path $\omega,$ by (\ref{2}) ${\sup_{0\leq s\leq T}}\tilde{Z}(s)\vee|\bar{\Theta}(W)(s)|\leq K(\omega),$ showing that  $\tau_{k}\uparrow T$\,\,$P$\,\,a.s. Thus, the assumptions on $L$ imply that
$$L(\tilde{Z}_{k_{}}(\cdot),W(\cdot))\longrightarrow L(\tilde{Z}(\cdot),W(\cdot))\quad \mbox{a.s.}$$   
In light of (\ref{sid}) and (\ref{ris}) the sequence $L(\tilde{Z}_{k_{}}(\cdot),W(\cdot))$ is uniform integrable, therefore  
$${\tilde{E}}L(\tilde{Z}_{k}(\cdot),W(\cdot)) \longrightarrow{\tilde{E}}L(
\tilde{Z}(\cdot),W(\cdot)).$$
 Next we claim that
 $$\int_{0}^{T}{\tilde{E}}(\lambda_{k_{}}(t)\vert\mathcal{F}_{t})d\bar{W}(t)\longrightarrow\int_{0}^{T}{\tilde{E}}(\lambda(t)\vert\mathcal{F}_{t})d\bar{W}(t),$$
$P$ a.s. Indeed, by It\^o's isometry it suffices to prove that
$${\tilde{E}}\int_{0}^{T}\|{\tilde{E}}((\lambda_{k}(t)-\lambda(t))\vert\mathcal{F}_{t})\|^{2}\,dt \longrightarrow 0.$$
Moreover $${\tilde{E}}\int_{0}^{T}\|{\tilde{E}}(\lambda_{k}(t)-\lambda(t)\vert\mathcal{F}_{t})\|^{2}\,dt\leq \int_{0}^{T}{\tilde{E}}\|\lambda_{k}(t)-\lambda(t)\|^2\,dt,$$
due to Jensen's inequality. Hence we want to prove
\begin{equation}\label{023}
\int_{0}^{T}{\tilde{E}}\|\lambda_{k}(t)-\lambda(t)\|^2\,dt\longrightarrow 0.
\end{equation}
Let us recall that
 $$\lambda_{k}(t,\omega):=\left[\int_{t}^{T}\mu(du,(\tilde{Z}_{k}(\omega),W(\omega)))\Phi_k(u,t,\omega)\right]g_{k}(t,(\tilde{Z}_{k}(\omega),W(\omega))),$$ and
 $$\lambda(t,\omega):=\left[\int_{t}^{T}\mu(du,(\tilde{Z}(\omega),W(\omega)))\Phi(u,t,\omega)\right]g(t,(\tilde{Z}(\omega),W(\omega))). $$
Because of $\tau_{k}\uparrow T$\,\,\,$P$\,\,a.s.\, and $\tilde{Q}\sim P,$ 
\begin{eqnarray*}\lefteqn{[g_{k}(t,(\tilde{Z}_{k}(\omega),W(\omega)))-g(t,(\tilde{Z}(\omega),W(\omega)))]}\\&&\mbox{}=[g_{k}(t,(\tilde{Z}_{k}(\omega),W(\omega)))-g(t,(\tilde{Z}(\omega),W(\omega)))]1_{\{\tau_{k}\leq t\}}\longrightarrow 0,\,\, \tilde{Q}\,\,a.s. \end{eqnarray*} 
In order to prove the claim it suffices to show
\begin{equation}\label{show1}
{\tilde{E}}\left\|\int_{t}^{T}\mu(du,(\tilde{Z}_{k},W))\Phi_k(u,t)-\int_{t}^{T}\mu(du,(\tilde{Z},W))\Phi(u,t)\right\|^{2}\rightarrow 0,
\end{equation} and
\begin{equation}\label{show2}
{\tilde{E}}\,\|\lambda_{k}(t)\|^{2+\epsilon}\leq K_{1}
,\end{equation} for some $\epsilon>0$ and a constant $K_1$ independent of $k$ and $t.$ Indeed (\ref{show1}) give the almost sure convergence (up to a subsequence) of $\lambda_{k}$ to $\lambda.$ Moreover (\ref{show2}) implies the uniform convergence of $\|\lambda_{k}\|^{2}$, and also yields (\ref{023}) by Lebesque Dominated Convergence Theorem. 
 To proceed, we need some bounds on $\Phi_{k}(t,\cdot),$ and $\Phi(t,\cdot)$ independent of $k$ and $t.$  Cf (\ref{oop})
\begin{equation}\label{oop1}
 \sup_{s\leq v\leq t}\parallel\Phi^{2,j}(v,s)\parallel\leq e^{K_{2}(t-s)}.
\end{equation}
Moreover $\Phi_{k}^{2j}=\Phi^{2j}$ and $\Phi^{21}\equiv0.$ 
Furthermore we prove for $j=1,\cdots,n+1$
\begin{equation}\label{112}
{\tilde{E}}\,[\,\sup_{0\leq t\leq T}|\Phi_{k}^{1j}(t,s)|]^{m}\leq K_3
,\end{equation}and

\begin{equation}\label{113}
{\tilde{E}}\,[\,\sup_{0\leq t\leq T}|\Phi^{1j}(t,s)|]^{m}\leq K_4
,\end{equation} 
for some constants $K_3,$ $K_4,$ and $m>1.$  We prove (\ref{113}), and (\ref{112}) follows similarly. These arguments conclude the proof.
%
%
%

\end{document}